\documentclass{amsart}
\usepackage{amssymb, amsmath, latexsym, mathabx, dsfont,tikz}
\usepackage{multirow}
\usepackage{setspace}
\usepackage{enumerate}
\usepackage{diagbox}
\usepackage{stmaryrd}
\usepackage{makecell}

\renewcommand{\baselinestretch}{\baselinestretch}
\renewcommand{\baselinestretch}{1.1}
\numberwithin{equation}{section}

\newtheorem{thm}{Theorem}[section]
\newtheorem{lem}[thm]{Lemma}
\newtheorem{cor}[thm]{Corollary}
\newtheorem{prop}[thm]{Proposition}

\theoremstyle{definition}

\theoremstyle{remark}

\numberwithin{equation}{section}

\newcommand{\ra}{{\ \longrightarrow \ }}

\newcommand{\ratwo}{{\ \overset{2}{\longrightarrow} \ }}
\newcommand{\nratwo}{{\ \overset{2}{\longarrownot\longrightarrow} \ }}
\newcommand{\gen}{\text{gen}}

\newcommand{\n}{{\mathbb N}}
\newcommand{\z}{{\mathbb Z}}
\newcommand{\q}{{\mathbb Q}}
\newcommand{\Mod}[1]{\ (\mathrm{mod}\ #1 )}

\begin{document}
\title[]{Tight universal triangular forms}

\author{Mingyu Kim}

\address{Department of Mathematics, Sungkyunkwan University, Suwon 16419, Korea}
\email{kmg2562@skku.edu}

\thanks{This work was supported by the National Research Foundation of Korea(NRF) grant funded by the Korea government(MSIT) (NRF-2021R1C1C2010133).}

\begin{abstract}
For a subset $S$ of nonnegative integers and a vector $\mathbf{a}=(a_1,\dots,a_k)$ of positive integers, let $V'_S(\mathbf{a})=\{ a_1s_1+\cdots+a_ks_k : s_i\in S\} \setminus \{0\}$.
For a positive integer $n$, let $\mathcal T(n)$ be the set of integers greater than or equal to $n$.
In this paper, we consider the problem of finding all vectors $\mathbf{a}$ satisfying $V'_S(\mathbf{a})=\mathcal T(n)$, when $S$ is the set of (generalized) $m$-gonal numbers and $n$ is a positive integer.
In particular, we completely resolve the case when $S$ is the set of triangular numbers.
\end{abstract}
\maketitle

\section{Introduction}
For a positive integer $m$ greater than or equal to 3, the polynomial
$$
P_m(x)=\frac{(m-2)x^2-(m-4)x}{2}
$$
is an integer-valued quadratic polynomial.
Note that the $s$-th $m$-gonal number is given by $P_m(s)$ for nonnegative integers $s$.
For a vector $\mathbf{a}=(a_1,a_2,\dots,a_k)$ of positive integers, a polynomial of the form
$$
p_m(\mathbf{a})=p_m(\mathbf{a})(x_1,\dots,x_k)=a_1P_m(x_1)+\cdots+a_kP_m(x_k)
$$
in variables $x_1,x_2,\dots,x_k$ is called {\it a $k$-ary $m$-gonal form} (or {\it a $k$-ary sums of generalized $m$-gonal numbers}).
We say an integer $N$ is represented by an $m$-gonal form $p_m(\mathbf{a})$ if the equation 
$$
p_m(\mathbf{a})(x_1,\dots,x_k)=N
$$
has an integer solution.
The smallest positive integer represented by $p_m(\mathbf{a})$ will be denoted by $\min(p_m(\mathbf{a}))$ and we call it {\it the minimum of $p_m(\mathbf{a})$}.
We call an $m$-gonal form {\it tight universal} if it represents every positive integer greater than its minimum.
A tight universal $m$-gonal form having minimum 1 is simply called {\it universal}.
The universal $m$-gonal forms have been studied by many mathematicians and there are several results on the classification problem (see, for example, \cite{BK},\cite{Ju5},\cite{JuOh8} and \cite{Dj}).
Note that $P_4(x)=x^2$ and the classification of universal diagonal quadratic forms can be easily done by using Conway-Schneeberger 15-Theorem(see \cite{B} and \cite{C}).

Recently, Oh and the author \cite{KO6} studied (positive definite integral) quadratic forms which represent every positive integer greater than the minimum of each quadratic forms.
We called such a quadratic form $f$ {\it tight $\mathcal T(n)$-universal}, where $n$ is the minimum of the quadratic form $f$.
In that paper, we classified ``diagonal" tight universal quadratic forms, i.e., the classification of tight universal $m$-gonal forms in the case of $m=4$ has already been done.

We follow the notations and terminologies used in \cite{KO6}.
For $n=1,2,3,\dots$, we denote by $\mathcal T(n)$ the set of integers greater than $n$.
We also say an $m$-gonal form is tight $\mathcal T(n)$-universal if it is tight universal with minimum $n$.
In Section 3, we resolve the classification problem of tight $\mathcal T(n)$-universal $m$-gonal forms in the following cases;
\begin{enumerate} [(i)]
\item $m=5$,\ \ $n\ge 7$;
\item $m=7$,\ \ $n\ge 11$;
\item $m\ge 8$,\ \ $n\ge 2m-5$.
\end{enumerate}
In fact, it will be proved that there are ``essentially" two tight $\mathcal T(n)$-universal $m$-gonal forms in the cases of (ii) and (iii).
It will also be shown that there is (essentially) unique tight $\mathcal T(n)$-universal pentagonal forms for any $n\ge 7$.
In addition, we classify tight $\mathcal T(n)$-universal sums of $m$-gonal numbers (for definition, see Section 3).
In Section 4, we classify tight universal triangular forms by finding all tight $\mathcal T(n)$-universal triangular forms for every integer $n\ge 3$.
Note that universal triangular forms were classified in \cite{BK} and tight $\mathcal T(2)$-universal triangular forms were found in \cite{Ju}.
To classify tight universal triangular forms, we use the theory of quadratic forms and adapt the geometric language of quadratic spaces and lattices, generally following \cite{Ki} and \cite{OM}.
Some basic notations and terminologies will be given in Section 2.

\section{Preliminaries}
Let $R$ be the ring of rational integers $\z$ or the ring of $p$-adic integers $\z_p$ for a prime $p$ and let $F$ be the field of fractions of $R$.
An \textit{$R$-lattice} is a finitely generated $R$-submodule of a quadratic space $(W,Q)$ over $F$.
We let $B:W\times W\to F$ be the symmetric bilinear form associated to the quadratic map $Q$ so that $B(x,x)=Q(x)$ for every $x\in W$.
For an element $a$ in $R$ and an $R$-lattice $L$, we say {\it $a$ is represented by $L$ over $R$} and write $a\ra L$ over $R$ if $Q(\mathbf{x})=a$ for some vector $\mathbf{x}\in L$.

Let $L$ be a $\z$-lattice on a quadratic space $W$ over $\q$.
The genus of $L$, denoted $\gen(L)$, is the set of all $\z$-lattices on $W$ which are locally isometric to $L$.
The number of isometry classes in $\gen(L)$ is called the class number of $L$ and denoted by $h(L)$.
If an integer $a$ is represented by $L$ over $\z_p$ for all primes $p$ (including $\infty$), then there is a $\z$-lattice $K$ in $\gen(L)$ such that $a\ra K$ (for this, see \cite[102:5]{OM}).
In this case, we say $a$ is represented by the genus of $L$ and write $a\ra \gen(L)$.
For a $\z$-basis $\{ \mathbf{v}_1,\mathbf{v}_2,\dots,\mathbf{v}_k\}$ of $L$, the corresponding quadratic form $f_L$ is defined as
$$
f_L=\sum_{i,j=1}^k B(\mathbf{v}_i,\mathbf{v}_j)x_ix_j.
$$
If $L$ admits an orthogonal basis $\{\mathbf{w}_1,\mathbf{w}_2,\dots,\mathbf{w}_k\}$, then we simply write
$$
L\simeq \langle Q(\mathbf{w}_1),Q(\mathbf{w}_2),\dots,Q(\mathbf{w}_k)\rangle.
$$
In this article, we abuse the notation and the diagonal quadratic form $a_1x_1^2+a_2x_2^2+\cdots+a_kx_k^2$ will also be denoted by $\langle a_1,a_2,\dots,a_k\rangle$.
The scale ideal of $L$ is denoted by $\mathfrak{s}(L)$.
Throughout the paper, we always assume that every $\z$-lattice is positive definite and primitive in the sense that $\mathfrak{s}(L)=\z$.
Any unexplained notations and terminologies on the representation of quadratic forms can be found in \cite{Ki} or \cite{OM}.

Throughout this section, $S$ always denotes a subset of nonnegative integers containing 0 and 1, unless otherwise stated.
For a vector $\mathbf{a}=(a_1,a_2,\dots,a_k)\in \n^k$, we define
$$
V_S(\mathbf{a})=\{ a_1s_1+a_2s_2+\cdots+a_ks_k : s_i\in S\}
$$
and define $V'_S(\mathbf{a})=V_S(\mathbf{a})-\{0\}$.
For example, if $S$ is the set of squares of integers, then
$$
V'_S(1,1,1,1)=\n,\quad V'_S(1,1,1)=\n-\left\{4^a(8b+7) : a,b\in \n_0\right\}
$$
by Lagrange's four-square theorem and Legendre's three-square theorem, respectively.
We denote the set of nonnegative integers by $\n_0$ for simplicity.
For two vectors $\mathbf{u}=(u_1,u_2,\dots,u_r)\in \n^r$ and $\mathbf{v}=(v_1,v_2,\dots,v_s)\in \n^s$, we write
$$
\mathbf{u}\preceq \mathbf{v}\ (\mathbf{u}\prec \mathbf{v})
$$
if $\{u_i\}_{1\le i\le r}$ is a subsequence (proper subsequence, respectively) of $\{v_j\}_{1\le j\le s}$.
Let $n$ be a positive integer and let $\mathbf{a}$ be a vector of positive integers.
We say {\it $\mathbf{a}$ is tight $\mathcal T(n)$-universal with respect to $S$} if $V'_S(\mathbf{a})=\mathcal T(n)$.
When $n=1$, we simply say {\it $\mathbf{a}$ is universal with respect to $S$}.
We say {\it $\mathbf{a}$ is new tight $\mathcal T(n)$-universal with respect to $S$} if $V'_S(\mathbf{a})=\mathcal T(n)$, and $V'_S(\mathbf{b})\subsetneq \mathcal T(n)$ whenever $\mathbf{b}\prec \mathbf{a}$.
For simplicity, we use the following notation throughout the article.
For $n_1,n_2,\dots,n_r\in \n$ and $e_1,e_2,\dots,e_r\in \n_0$, we denote by 
$$
\mathbf{n_1}^{e_1}\mathbf{n_2}^{e_2}\cdots \mathbf{n_r}^{e_r}
$$
the vector
$$
(n_1,\dots,n_1,n_2,\dots,n_2,\dots,n_r,\dots,n_r)\in \z^{e_1+e_2+\cdots+e_r},
$$
where each $n_i$ is repeated $e_i$ times for $i=1,2,\dots,r$.

\begin{lem} \label{lemeasy}
Let $\mathbf{a},\mathbf{b}$ be vectors of positive integers such that $\mathbf{a}\preceq \mathbf{b}$ and let $S, S'$ be subsets of nonnegative integers containing 0 and 1 such that $S\subseteq S'$.
Then we have the followings;
\begin{enumerate} [(i)]
\item $V_S(\mathbf{a})\subseteq V_S(\mathbf{b})$;
\item $V_S(\mathbf{a})\subseteq V_{S'}(\mathbf{a})$;
\item $V_S(u+v)\subset V_S(u,v)$ for any $u,v\in \n$;
\item $\min(V'_S(\mathbf{a}))=\min\{a_i : 1\le i\le k\}$, where $\mathbf{a}=(a_1,a_2,\dots,a_k)$.
\end{enumerate}
\end{lem}
\begin{proof}
Trivial.
\end{proof}

\begin{lem} \label{lem123}
Let $\mathbf{a}=\mathbf{1}^{e_1}\mathbf{2}^{e_2}\mathbf{3}^{e_3}$ be a vector with a positive integer $e_1$ and nonnegative integers $e_2$ and $e_3$.
Assume that $V_S(\mathbf{a})=\n_0$.
Then for any integer $n\ge 2e_3+3$, the vector
$$
\mathbf{b}=\mathbf{n}^{e_1}\mathbf{n+1}^1\mathbf{n+2}^1\cdots \mathbf{2n-1}^1\mathbf{2n}^{e_2}
$$
is tight $\mathcal T(n)$-universal with respect to $S$.
\end{lem}
\begin{proof}
Let $n$ be an integer with $n\ge 2e_3+3$ and let $m$ be an integer greater than or equal to $n$.
Then $m$ can be written in the form $un+v$ for a nonnegative integer $u$ and an integer $v$ with $n\le v\le 2n-1$.
To prove the lemma, it suffices to show that $un+v\in V_S(\mathbf{b})$.
Since
$$
u\in \n_0 =V_S(\mathbf{a})=V_S(\mathbf{1}^{e_1}\mathbf{2}^{e_2}\mathbf{3}^{e_3}),
$$
we have 
$$
un\in V_S(\mathbf{n}^{e_1}\mathbf{2n}^{e_2}\mathbf{3n}^{e_3}).
$$
On the other hand, by applying Lemma \ref{lemeasy}(iii) $e_3$ times, we have
$$
V_S(\mathbf{3n}^{e_3})\subseteq V_S(n+1,2n-1,n+2,2n-2,\dots,\widehat{v},\widehat{3n-v},\dots,n+e_3+1,2n-e_3-1),
$$
where the hat symbol $\hat{}$ indicates that the component is omitted.
From these follows that
\begin{align*}
un&\in V_S(\mathbf{n}^{e_1}\mathbf{2n}^{e_2}\mathbf{3n}^{e_3})\\
&\subseteq V_S(\mathbf{n}^{e_1}\mathbf{2n}^{e_2}\mathbf{n+1}^1\mathbf{2n-1}^1\cdots \widehat{\mathbf{v}^1}\widehat{\mathbf{3n-v}^1}\cdots \mathbf{n+e_3+1}^1\mathbf{2n-e_3-1}^1).
\end{align*}
Therefore, we have
\begin{align*}
un+v&\in V_S(\mathbf{n}^{e_1}\mathbf{2n}^{e_2}\mathbf{n+1}^1\mathbf{2n-1}^1\cdots \mathbf{v}^1\mathbf{3n-v}^1\cdots \mathbf{n+e_3+1}^1\mathbf{2n-e_3-1}^1)\\
&\subseteq V_S(\mathbf{n}^{e_1}\mathbf{n+1}^1\mathbf{n+2}^1\cdots \mathbf{2n-1}^1\mathbf{2n}^{e_2}).
\end{align*}
This completes the proof.
\end{proof}

For $n=1,2,3,\dots$, we define vectors $\mathbf{x}_n, \mathbf{y}_n\in \z^{n+1}$ as
$$
\mathbf{x}_n=(n,n,n+1,n+2,\cdots,2n-1),\quad \mathbf{y}_n=(n,n+1,n+2,\cdots,2n).
$$

\begin{lem} \label{lemxy}
Let $n$ be a positive integer and let $\mathbf{a}=(a_1,a_2,\dots,a_k)\in \n^k$ with $a_1\le a_2\le \cdots \le a_k$ such that $V'_S(\mathbf{a})=\mathcal T(n)$. Then we have $(n,n+1,n+2,\dots,2n-1)\preceq \mathbf{a}$.
Furthermore, if $2\not\in S$, then $\mathbf{x}_n\preceq \mathbf{a}$ or $\mathbf{y}_n\preceq \mathbf{a}$.
\end{lem}
\begin{proof}
Since $V'_S(\mathbf{a})=\mathcal T(n)$, we have
\begin{equation} \label{eqxy1}
n=a_1\le a_2\le \cdots \le a_k.
\end{equation}
Thus, to prove the first assertion, if suffices to show that for any integer $v$ with $n+1\le v\le 2n-1$, there is an integer $j_v$ with $1\le j_v\le k$ such that $a_{j_v}=v$.
Let $v$ be an integer such that $n+1\le v\le 2n-1$.
Since $v\in V'_S(\mathbf{a})$, we have
$$
v=a_1s_1+a_2s_2+\cdots+a_ks_k
$$
for some $s_1,s_2,\dots,s_k\in S$.
Since $v>0$, there is an integer $j_v$ with $1\le j_v\le k$ such that $s_{j_v}>0$.
If $s_l>0$ for some $l$ different form $j_v$, then
$$
v=a_1s_1+a_2s_2+\cdots+a_ks_k\ge a_{j_v}s_{j_v}+a_ls_l\ge a_{j_v}+a_l\ge 2n
$$ by Equation \ref{eqxy1},
and this is absurd since $v\le 2n-1$.
It follows that $s_{j_v}=1$ and $s_l=0$ for any $l\neq j_v$.
Thus we have $v=a_{j_v}$ and the first assertion follows.

Now we assume further that $2\not\in S$.
Then clearly we have
$$
2n\in V_S(\mathbf{a})-V_S(n,n+1,n+2,\dots,2n-1).
$$
From this one may easily deduce that
$$
(n,n,n+1,n+2,\dots,2n-1)\preceq \mathbf{a}\ \ \text{or}\ \ (n,n+1,n+2,\dots,2n-1,2n)\preceq \mathbf{a}.
$$
This completes the proof.
\end{proof}

\section{Tight $\mathcal T(n)$-universal sums of (generalized) $m$-gonal numbers}
Let $m$ be an integer greater than or equal to 3.
We denote the set of (generalized) $m$-gonal numbers by $\mathcal{P}_m$ ($\mathcal{GP}_m$, respectively),
i.e.,
$$
\mathcal{P}_m=\left\{ \frac{(m-2)x^2-(m-4)x}2 : x\in \n_0 \right\}
$$
and
$$
\mathcal{GP}_m=\left\{ \frac{(m-2)x^2-(m-4)x}2 : x\in \z \right\}.
$$
One may easily check the followings;
\begin{enumerate} [(i)]
\item $\{0,1\} \subset \mathcal{P}_m\subseteq \mathcal{GP}_m$ for any $m\ge 3$;
\item $2\not\in \mathcal{P}_m$ for any $m\ge 3$;
\item $2\in \mathcal{GP}_m$ only if $m=5$;
\item $\mathcal{P}_3=\mathcal{GP}_3=\mathcal{GP}_6$;
\item $\mathcal{P}_4=\mathcal{GP}_4$.
\end{enumerate}

\begin{prop} \label{propgpoly}
Let $m$ be an integer greater than or equal to 8.
If $n\ge 2m-5$, then both $\mathbf{x}_n$ and $\mathbf{y}_n$ are tight $\mathcal T(n)$-universal with respect to $\mathcal{GP}_m$.
\end{prop}
\begin{proof}
By \cite[Theorem 1.1]{S} and \cite[Theorem 3.2]{Dj}, we have
$$
V_{\mathcal{GP}_m}(\mathbf{1}^{m-4})=\n_0.
$$
From this, one may easily deduce that
$$
V_{\mathcal{GP}_m}(\mathbf{1}^{e_1}\mathbf{2}^{e_2}\mathbf{3}^{m-4})=\n_0
$$
for $(e_1,e_2)\in \{(2,0),(1,1)\}$.
Now the proposition follows immediately from Lemma \ref{lem123}.
\end{proof}

\begin{thm} \label{thmgpoly}
Let $m$ be an integer greater than or equal to 8.
If $n\ge 2m-5$, then there are exactly two new tight $\mathcal T(n)$-universal $m$-gonal forms.
\end{thm}
\begin{proof}
Note that $2\not\in \mathcal{GP}_m$ since $m\neq 5$.
The theorem follows immediately from the second assertion of Lemma \ref{lemxy} and Proposition \ref{propgpoly}.
\end{proof}

\begin{prop}
There is only one new tight $\mathcal T(n)$-universal pentagonal forms for any $n\ge 7$.
\end{prop}
\begin{proof}
Note that the vector $(1,3,3)$ is universal with respect to $\mathcal{GP}_5$ (see, \cite{GS}).
By Lemma \ref{lem123}, the vector
$$
(n,n+1,n+2,\dots,2n-1)
$$
is tight $\mathcal T(n)$-universal with respect to $\mathcal{GP}_5$ for any $n\ge 7$.
Now the proposition follows immediately from the first assertion of Lemma \ref{lemxy}.
\end{proof}

\begin{prop}
There are exactly two new tight $\mathcal T(n)$-universal heptagonal forms for any $n\ge 11$.
\end{prop}
\begin{proof}
Note that $V_{\mathcal{GP}_7}(1,1,1,1)=\n_0$ (see \cite{S} or \cite[Theorem 1.2]{Dj}).
It follows that
$$
V_{\mathcal{GP}_7}(\mathbf{1}^{e_1}\mathbf{2}^{e_2}\mathbf{3}^4)=\n_0
$$
for $(e_1,e_2)\in \{(2,0),(1,1)\}$.
The proposition follows immediately from Lemma \ref{lem123} and the second assertion of Lemma \ref{lemxy}.
\end{proof}

Let $n$ be a positive integer.
Now we define (new) tight $\mathcal T(n)$-universal sums of $m$-gonal numbers.
For an integer $m\ge 3$ and a vector $\mathbf{a}$ of positive integers,
we call the pair $(\mathcal P_m,\mathbf{a})$ {\it a sum of $m$-gonal numbers}.
We say $(\mathcal P_m,\mathbf{a})$ is tight $\mathcal T(n)$-universal if $V'_{\mathcal P_m}(\mathbf{a})=\mathcal T(n)$.
A tight $\mathcal T(n)$-universal sum of $m$-gonal numbers $(\mathcal P_m,\mathbf{a})$ is called {\it new} if $(\mathcal P_m,\mathbf{b})$ is not $\mathcal T(n)$-universal whenever $\mathbf{b}\prec \mathbf{a}$, or equivalently, $V'_{\mathcal P_m}(\mathbf{b})\subsetneq \mathcal T(n)$ whenever $\mathbf{b}\prec \mathbf{a}$.

\begin{prop} \label{proppoly}
Let $m$ be an integer greater than or equal to 3.
If $n\ge 2m+3$, then both $(\mathcal P_m,\mathbf{x}_n)$ and $(\mathcal P_m,\mathbf{y}_n)$ are tight $\mathcal T(n)$-universal.
\end{prop}
\begin{proof}
Fermat polygonal number theorem says that
$$
V_{\mathcal{P}_m}(\mathbf{1}^m)=\n_0.
$$
From this, one may easily deduce that
$$
V_{\mathcal{P}_m}(\mathbf{1}^{e_1}\mathbf{2}^{e_2}\mathbf{3}^m)=\n_0
$$
for $(e_1,e_2)\in \{(2,0),(1,1)\}$.
Now the tight $\mathcal T(n)$-universalities (with respect to $\mathcal{P}_m$) of $\mathbf{x}_n$ and $\mathbf{y}_n$ follows immediately from Lemma \ref{lem123}.
\end{proof}

\begin{thm} \label{thmpoly}
Let $m$ be an integer greater than or equal to 3.
If $n\ge 2m+3$, then there are exactly two new tight $\mathcal T(n)$-universal sums of $m$-gonal numbers.
\end{thm}
\begin{proof}
Note that $2\not\in \mathcal{P}_m$.
The theorem follows immediately from the second assertion of Lemma \ref{lemxy} and Proposition \ref{proppoly}.
\end{proof}

\section{Tight universal triangular forms}
In this section, we classify tight universal triangular forms.
As noted in Introduction, for $n=1,2$, tight $\mathcal T(n)$-universal triangular forms were classified in \cite{BK} and \cite{Ju}, respectively.
We first prove that there are exactly 12 new tight $\mathcal T(3)$-universal triangular forms as listed in Table \ref{tableg3}.
We also prove that there are exactly two new tight $\mathcal T(n)$-universal triangular forms
$$
X_n=p_3(n,n,n+1,n+2,\dots,2n-1)\ \ \text{and}\ \ Y_n=p_3(n,n+1,n+2,\dots,2n-1,2n),
$$
for any $n\ge 4$.
We introduce some notations which will be used throughout this section.
Recall that a triangular form is a polynomial of the form
$$
p_3(a_1,a_2,\dots,a_k)=a_1\frac{x_1(x_1+1)}{2}+\cdots+a_k\frac{x_k(x_k+1)}{2},
$$
where $(a_1,a_2,\dots,a_k)$ is a vector of positive integers.
For a nonnegative integer $g$ and a triangular form $p_3(a_1,a_2,\dots,a_k)$, we write
$$
g\ra p_3(a_1,a_2,\dots,a_k)
$$
if $g$ is represented by $p_3(a_1,a_2,\dots,a_k)$.
For a positive integer $u$ and a diagonal quadratic form $\langle a_1,a_2,\dots,a_k\rangle$, we write
$$
u\ratwo \langle a_1,a_2,\dots,a_k\rangle
$$
if there is a vector $(x_1,x_2,\dots,x_k)\in \z^k$ with $(2,x_1x_2\cdots x_k)=1$ such that 
$$
a_1x_1^2+a_2x_2^2+\cdots+a_kx_k^2=u.
$$
One may easily show the following observation which will be used to show the tight universalities of triangular forms:
A nonnegative integer $g$ is represented by a triangular form $p_3(a_1,a_2,\dots,a_k)$ if and only if
$$
8g+a_1+a_2+\cdots+a_k\ratwo \langle a_1,a_2,\dots,a_k\rangle.
$$
We also need the notion of regularities of ternary triangular forms.
A ternary triangular form $p_3(a,b,c)$ is called {\it regular} if, for every nonnegative integer $g$, the following holds:
If $8g+a+b+c\ra \langle a,b,c\rangle$ over $\z_p$ for every odd prime $p$, then $8g+a+b+c\ratwo \langle a,b,c\rangle$.
For more information about regular ternary triangular forms, we refer the reader to \cite{KO3}.

\begin{table}[ht]
\caption{New tight $\mathcal T(3)$-universal triangular forms $p_3(a_1,a_2,\dots,a_k)$} 
\label{tableg3}
	\begin{center}
		\begin{tabular}{|ccccc|c|}
			\Xhline{1pt}
			$a_1$&$a_2$&$a_3$&$a_4$&$a_5$& Conditions on $a_5$ \\
			\hline
			3&3&4&5&&\\
			3&4&4&5&6&\\
			3&4&5&5&6&\\
			3&4&5&6&$a_5$&$6\le a_5\le 16,\ a_5\neq 14,15$\\
			\Xhline{1pt}
		\end{tabular}
	\end{center}
\end{table}

\begin{prop} \label{propx3}
The quaternary triangular form $X_3=p_3(3,3,4,5)$ is tight $\mathcal T(3)$-universal.
\end{prop}

\begin{proof}
One may directly check that $X_3$ represents all integers from 3 to 14.
Let $g$ be a positive integer greater than 14 and put $g'=8g+15$.
To show that $g$ is represented by $X_3$, it suffices to show that 
$$
g'\ratwo \langle 3,3,4,5\rangle.
$$
Define sets $A$ and $B$ by
\begin{align*}
A&=\left\{ u\in \n : u\equiv 1\Mod 3 \ \text{or}\ u\equiv 3,6\Mod 9 \right\}, \\
B&=\left\{ u\in \n : u\equiv 2\Mod 8,\ u\ge 10\right\}.
\end{align*}
We assert that
$$
v\ratwo \langle 3,3,4\rangle
$$
for any $v\in A\cap B$.
To show the assertion, let $v\in A\cap B$.
One may easily check that every positive integer in $A$ is represented by the diagonal quadratic form $\langle 3,3,4\rangle$ over $\z_3$.
Note that $\langle 3,3,4\rangle$ represents all elements in $\z_p$ over $\z_p$ for any prime $p\ge 5$.
Thus we have
$$
v\ra \langle 3,3,4\rangle
$$
over $\z_p$ for all odd primes $p$.
Furthermore, $v=8v'+10$ for some nonnegative integer $v'$ since $v\in B$. From these and the fact that the ternary triangular form $p_3(3,3,4)$ is regular (for this, see \cite{KO3}), it follows that
$$
v\ratwo \langle 3,3,4\rangle.
$$
So we have the assertion.
If we define an odd positive integer $d$ by
$$
d=\begin{cases} 1&\text{if}\ \ g'\equiv 0\Mod 3 \ \ \text{or}\ \ g'\equiv 2,8\Mod 9 ,\\
3&\text{if}\ \ g'\equiv 1\Mod 3 ,\\
5&\text{if}\ \ g'\equiv 5\Mod 9 .\end{cases}
$$
then one may easily check that $g'-5d^2\in A\cap B$.
Thus we have
$$
g'-5d^2\ratwo \langle 3,3,4\rangle.
$$
Since $d$ is odd, it follows that
$$
g'\ratwo \langle 3,3,4,5\rangle.
$$
This completes the proof.
\end{proof}

We use the following lemma proved by B.W. Jones in his unpublished thesis \cite{J}.

\begin{lem}[Jones] \label{lemJones}
Let $p$ be an odd prime and $k$ be a positive integer not divisible by $p$ such that the Diophantine equation $x^2+ky^2=p$ has an integer solution.
If the Diophantine equation 
$$
x^2+ky^2=N\ (N>0)
$$
has an integer solution, then it also has an integer solution $(x_0,y_0)$ satisfying
$$
\gcd(x_0,y_0,p)=1.
$$
\end{lem}

\begin{prop} \label{prop346}
Let $g$ be a positive integer congruent to 5 modulo 8.
Assume that $g$ is congruent to 1 modulo 3 or is a multiple of 9.
Then $g$ is represented by the diagonal ternary quadratic form $3x^2+4y^2+6z^2$.
\end{prop}

\begin{proof}
Let $L=\langle 3,4,6\rangle$.
Note that the class number $h(L)$ of $L$ is two and the genus mate is $\langle 1,6,12\rangle$.
From the assumptions, one may easily check that
$$
g\ra \gen(\langle 3,4,6\rangle).
$$
We may assume that $g\ra \langle 1,6,12\rangle$ since otherwise we are done.
Thus there is a vector $(x_1,y_1,z_1)\in \z^3$ such that
$$
g=x_1^2+6y_1^2+12z_1^2.
$$

First, assume that $g\equiv 0\Mod 9$.
One may easily check that $x_1\equiv 0\Mod 3$, and that $y_1\equiv 0\Mod 3$ if and only if $z_1\equiv 0\Mod 3$.
By changing sign of $z_1$ if necessary, we may further assume that $y_1\equiv z_1\Mod 3$.
Thus $x_1=3x_2$ and $y_1=z_1-3y_2$ with integers $x_2$ and $y_2$.
Now
\begin{align*}
g&=x_1^2+6y_1^2+12z_1^2\\
&=(3x_2)^2+6(z_1-3y_2)^2+12z_1^2\\
&=3(x_2+2y_2-2z_1)^2+4(3y_2)^2+6(x_2-y_2+z_1)^2.
\end{align*}

Second, assume that $g\equiv 1\Mod 3$.
If $y_1^2+2z_1^2=0$, then $g=x_1^2$ and this is absurd since $g\equiv 5\Mod 8$.
Hence $y_1^2+2z_1^2\neq 0$ and thus by Lemma \ref{lemJones}, there are integers $y_3$ and $z_3$ with $\gcd(y_3,z_3,3)=1$ such that
$$
y_1^2+2z_1^2=y_3^2+2z_3^2.
$$
Note that $x_1\not\equiv 0\Mod 3$ since $g\equiv 1\Mod 3$.
After changing signs of $y_3$ and $z_3$ if necessary, we may assume that $x_1+y_3+2z_3\equiv 0\Mod 3$.
Then
\begin{align*}
g&=x_1^2+6y_3^2+12z_3^2\\
&=3\left( \frac{x_1-2y_2-4z_2}{3}\right)^2+4\left( y_2-z_2\right)^2+6\left( \frac{x_1+y_3+2z_3}{3}\right)^2.
\end{align*}
Since $x_1-2y_3-4z_3\equiv x_1+y_3+2z_3\equiv 0\Mod 3$, we have $g\ra L$.
This completes the proof.
\end{proof}
\begin{prop} \label{propy3}
The triangular form $Y_3=p_3(3,4,5,6)$ represents all positive integers but 1,2 and 16.
\end{prop}

\begin{proof}
One may directly check that $Y_3$ represents all integers from 3 to 29 except 16.
Let $g$ be an integer greater than 29 and put $g'=8g+18$.
If we define an odd positive integer $d$ by
$$
d=\begin{cases} 1&\text{if}\ \ g'\equiv 0\Mod 3 \ \ \text{or}\ \ g'\equiv 5\Mod 9 ,\\
3&\text{if}\ \ g'\equiv 1\Mod 3 ,\\
5&\text{if}\ \ g'\equiv 8\Mod 9 ,\\
7&\text{if}\ \ g'\equiv 2\Mod 9 ,\end{cases}
$$
then one may easily check that $g'-5d^2\equiv 1\Mod 3$ or $g'-5d^2\equiv 0\Mod 9$.
Furthermore, we have $g'-5d^2\equiv 5\Mod 8$ since $d$ is odd.
Hence we have $g'-5d^2\ra \langle 3,4,6\rangle$ by Proposition \ref{prop346}.
Thus there is a vector $(x,y,z)\in \z^3$ such that
$$
g'-5d^2=3x^2+4y^2+6z^2.
$$
One may easily deduce from $g'-5d^2\equiv 5\Mod 8$ that $xyz\equiv 1\Mod 2$.
Thus $g'\ratwo \langle 3,4,5,6\rangle$.
This completes the proof.
\end{proof}

\begin{cor}
All of the quinary triangular forms in Table \ref{tableg3} are tight $\mathcal T(3)$-universal.
\end{cor}

\begin{proof}
Let $Z=p_3(a_1,a_2,a_3,a_4,a_5)$ be any quinary triangular form in Table \ref{tableg3}.
One may see that
$$
(3,4,5,6)\prec (a_1,a_2,a_3,a_4,a_5).
$$
From this and Proposition \ref{propy3} follows that $Z$ represents every integer greater than or equal to 3 except 16.
One may directly check that $Z$ also represents 16.
This completes the proof.
\end{proof}

\begin{prop}
Every new tight $\mathcal T(3)$-universal triangular form appears in Table \ref{tableg3}.
\end{prop}

\begin{proof}
Let $p_3=p_3(a_1,a_2,\dots,a_k)$ be a new tight $\mathcal T(3)$-universal triangular form.
By Lemma \ref{lemxy}, we have $X_3\preceq p_3$ or $Y_3\preceq p_3$.

First, assume that $X_3\preceq p_3$.
From the fact that $X_3$ is tight $\mathcal T(3)$-universal and the assumption that $p_3$ is new tight $\mathcal T(3)$-universal, it follows that $p_3=X_3$.

Second, assume that $Y_3\preceq p_3$.
Since $Y_3$ is not $\mathcal T(3)$-universal, it follows that $k>4$ and there is a vector $(j_1,j_2,j_3,j_4)\in \z^4$ with
$$
(j_1,j_2,j_3,j_4)\prec (1,2,\dots,k)
$$
such that $(a_{j_1},a_{j_2},a_{j_3},a_{j_4})=(3,4,5,6)$.
We put
$$
A=\{u\in \n : 3\le u\le 16, u\neq 14,15\}.
$$
If $a_j\not\in A$ for every $j\in \{1,2,\dots,k\} \setminus \{j_1,j_2,j_3,j_4\}$, then one may easily show that $p_3$ cannot represent 16, which is absurd.
Thus there is an integer $j$ with
$$
j\in \{1,2,\dots,k\} \setminus \{j_1,j_2,j_3,j_4\}
$$
such that $a_j\in A$.
One may check that $p_3(a_{j_1},a_{j_2},a_{j_3},a_{j_4},a_j)$ is in Table \ref{tableg3} and thus it is tight $\mathcal T(3)$-universal.
It follows that $k=5$ and $p_3=p_3(a_{j_1},a_{j_2},a_{j_3},a_{j_4},a_j)$ since otherwise $p_3$ is not new.
This completes the proof.
\end{proof}

\begin{thm} \label{thmxn}
For any integer $n$ greater than or equal to 3, the triangular form $X_n=p_3(n,n,n+1,n+2,\dots,2n-1)$ is tight $\mathcal T(n)$-universal.
\end{thm}

\begin{proof}

First, assume that $n\ge 6$.
Let $g$ be an integer greater than or equal to $n$.
Then $g$ can be written in the form $g=un+v$ for some nonnegative integer $u$ and an integer $v$ with $n\le v\le 2n-1$.
Note that the ternary triangular form $p_3(1,1,4)$ is universal and thus it represents $u$.
Thus $un$ is represented by $p_3(n,n,4n)$.
It follows that $un$ is represented by $p_3(n,n,n+1,n+2,2n-3)$.
Thus if $v\not\in \{n+1,n+2,2n-3\}$, then $un+v$ is represented by $p_3(n,n,n+1,n+2,2n-3,v)$ and thus by $X_n$.
On the other hand, the ternary triangular form $p_3(1,1,5)$ is also universal.
Hence $un$ is represented by $p_3(n,n,5n)$ and thus represented by $p_3(n,n,n+3,2n-2,2n-1)$ also.
This implies that if $v\in \{n+1,n+2,2n-3\}$, then $un+v$ is represented by $p_3(n,n,n+3,2n-2,2n-1,v)$ and thus by $X_n$.

Second, assume that $n=5$.
Let $g$ be an integer greater than or equal to 236.
We write $g=15u+v$, where $u$ is a positive integer and $v$ is an integer such that $0\le v\le 14$.
Note that the ternary triangular form $p_3(1,1,3)$ is regular.
For any nonnegative integer $w$, both $8\cdot 3w+5$ and $8(3w+1)+5$ are represented by $\langle 1,1,3\rangle$ over $\z_3$.
Thus $p_3(1,1,3)$ represents every nonnegative integer not equivalent to 2 modulo 3.
It follows that $p_3(5,5,6+9)$ represents every nonnegative integer congruent to 0 or 5 modulo 15.
Hence if $v\in \{0,5\}$, then 
$$
g=15u+v\ra p_3(5,5,6+9)
$$
and we have
$$
g\ra p_3(5,5,6,9).
$$
One may directly check that the binary triangular form $p_3(7,8)$ represents all integers in the set
$$
\{31,122,48,94,80,231,7,8,24\}.
$$
If $v\not\in \{0,5\}$, then one may easily see that there is a positive integer $a$ in the above set such that $g-a$ is a nonnegative integer congruent to 0 or 5 modulo 15.
Thus we have $g-a\ra p_3(5,5,6+9,7,8)$.
One may directly check that $p_3(5,5,6,7,8,9)$ represents all integers from 5 to 236.

Third, assume that $n=4$.
Note that the ternary triangular form $p_3(2,2,3)$ is regular.
From this one may easily show that it represents every nonnegative integer not congruent to 1 modulo 3.
Thus $p_3(4,4,6)$ represents every nonnegative integer of the form $6u$ and $6u+4$, where $u\in \z_{\ge 0}$.
Note that $p_3(5,7)$ represents 5,7,15 and 26 as
$$
5=5\cdot 1+7\cdot 0,\ 7=5\cdot 0+7\cdot 1,\ 15=5\cdot 3+7\cdot 0,\ 26=5\cdot 1+7\cdot 3.
$$
From this and the fact that $p_3(4,4,6)$ represents every nonnegative integer of the form $6u$, it follows that $p_3(4,4,5,6,7)$ represents every nonnegative integer of the form
$$
6u+7,\ 6u+26,\ 6u+15\ \ \text{and}\ \ 6u+5.
$$
One may directly check that $p_3(4,4,5,6,7)$ represents all integers from 4 to 25.

The case of $n=3$ was already proved in Proposition \ref{propx3}.
This completes the proof.
\end{proof}
\begin{thm} \label{thmyn}
For any integer $n$ greater than or equal to 4, the triangular form $Y_n=p_3(n,n+1,n+2,\dots,2n)$ is tight $\mathcal T(n)$-universal.
\end{thm}

\begin{proof}
First, assume that $n$ is greater than 4.
Let $g$ be an integer greater than  or equal to $n$.
We write $g=un+v$ for some  nonnegative integer  $u$ and an integer $v$ with $n\le v\le 2n-1$.
Since $n\ge 5$, there is an integer $n_1$ with $1\le n_1\le \left[ \frac n2\right]$  such that the three integers $n+n_1$, $2n-n_1$ and $v$ are all distinct.
Since the ternary triangular form $p_3(1,2,3)$ is universal, every nonnegative integer which is a multiple of $n$ is represented by $p_3(n,2n,3n)$ and thus also by $p_3(n,2n,n+n_1,2n-n_1)$.
It follows that $g=un+v$ is represented by $p_3(n,2n,n+n_1,2n-n_1,v)$.
From this and the choice of $v$, it follows that $g$ is represented by $Y_n$.

Now we assume that $n=4$.
Let $g_1$ be an integer greater than or equal to 830.
If we define two odd positive integers $\alpha$ and $\beta$ as
$$
(\alpha,\beta)=\begin{cases} (1,1)&\text{if}\ \ g_1\equiv 0\Mod 6,\\
(1,17)&\text{if}\ \ g_1\equiv 1\Mod 6,\\
(3,43)&\text{if}\ \ g_1\equiv 2\Mod 6,\\
(3,27)&\text{if}\ \ g_1\equiv 3\Mod 6,\\
(1,33)&\text{if}\ \ g_1\equiv 4\Mod 6,\\
(5,37)&\text{if}\ \ g_1\equiv 5\Mod 6,\end{cases}
$$
then one may easily check that $8g_1+30-5\alpha^2-7\beta^2$ is a nonnegative integer congruent to 18 modulo 48.
Put
$$
s=8g_1+30-5\alpha^2-7\beta^2
$$
and let $L=\langle 4,6,8\rangle$.
We assert that $s\ratwo L$.
One may easily check that $s$ is locally represented by $L$.
Note that the class number of $L$ is two and the genus mate is $M=\langle 2,4,24\rangle$.
Since $s$ is locally represented by $L$, we may assume that $s\ra M$.
Thus there is a vector $(x,y,z)\in \z^3$ such that
$$
s=2x^2+4y^2+24z^2.
$$
Since $s\equiv 0\Mod 3$, either $xy\not\equiv 0\Mod 3$ or $x\equiv y\equiv 0\Mod 3$ holds.
After changing sign of $y$ if necessary, we may assume that $x\equiv y\Mod 3$.
If we put $x=y-3x_1$, then
\begin{align*}
s&=2x^2+4y^2+24z^2\\
&=2(y-3x_1)^2+4y^2+24z^2\\
&=4(x_1+2z)^2+6(x_1-y)^2+8(x_1-z)^2.
\end{align*}
In the above equation, one may easily deduce
$$
x_1+2z\equiv x_1-y\equiv x_1-z\equiv 1\Mod 2
$$
from the fact that $s\equiv 2\Mod {16}$.
Thus we have $s\ratwo L$.
It follows immediately from this that
$$
8g_1+30\ratwo \langle 4,5,6,7,8\rangle,
$$
which is equivalent to $g_1\ra Y_4$.
On the other hand, one may directly check that $Y_4$ represents all integers from 4 to 829.
This completes the proof.
\end{proof}

\begin{thm} \label{thmg4}
For any integer $n$ exceeding 3, there are exactly two new tight $\mathcal T(n)$-universal triangular forms $X_n$ and $Y_n$.
\end{thm}
\begin{proof}
It follows immediately from Lemma \ref{lemxy} and Theorems \ref{thmxn},\ref{thmyn}.
\end{proof}


\end{document}